\documentclass[reqno,centertags]{amsart}

\usepackage{latexsym,upref,enumerate}
\usepackage{enumitem} 
\usepackage{nicefrac}
\usepackage{hyperref}

\newcommand*{\mailto}[1]{\href{mailto:#1}{\nolinkurl{#1}}}

\newtheorem{theorem}{Theorem}[section]

\newtheorem{lemma}[theorem]{Lemma}

\theoremstyle{definition}

\newtheorem{remark}[theorem]{Remark}
\newtheorem{example}[theorem]{Example}

\newcommand{\R}{{\mathbb R}}
\newcommand{\C}{{\mathbb C}}

\newcommand{\nn}{\nonumber}
\newcommand{\be}{\begin{equation}}
\newcommand{\ee}{\end{equation}} 
\newcommand{\id}{{\rm 1\hspace{-0.6ex}l}}

\newcommand{\E}{\mathrm{e}}

\newcommand{\im}{\mathrm{Im}}
\newcommand{\re}{\mathrm{Re}}
\newcommand{\dom}[1]{\mathrm{dom}\left(#1\right)}

\newcommand{\OO}{\mathcal{O}}
\newcommand{\oo}{o}
\newcommand{\ti}{\tilde}

\newcommand{\Tmax}{T_{\mathrm{max}}}

\newcommand{\Deftau}{\mathfrak{D}_\tau}
\newcommand{\ACloc}{AC_{\mathrm{loc}}[a,b)}

\newcommand{\eps}{\varepsilon}

\numberwithin{equation}{section}

\begin{document}

\title[Asymptotics of the Weyl Function]{Asymptotics of the Weyl Function for Schr\"odinger Operators 
with Measure-Valued Potentials}

\author[A.\ Luger]{Annemarie Luger}
\address{Department of Mathematics\\ Stockholm University\\ SE-106 91 Stockholm \\Sweden}
\email{\mailto{luger@math.su.se}}

\author[G.\ Teschl]{Gerald Teschl}
\address{Faculty of Mathematics\\ University of Vienna\\
Oskar-Morgenstern-Platz 1\\ 1090 Wien\\ Austria\\ and International
Erwin Schr\"odinger
Institute for Mathematical Physics\\ Boltzmanngasse 9\\ 1090 Wien\\ Austria}
\email{\mailto{Gerald.Teschl@univie.ac.at}}
\urladdr{\url{http://www.mat.univie.ac.at/~gerald/}}

\author[T.\ W\"ohrer]{Tobias W\"ohrer}
\address{Institute for Analysis und Scientific Computing\\Vienna University of Technology\\
Wiedner Hauptstra{\ss}e.\ 8--10/101\\ 1040 Vienna\\ Austria}
\email{\mailto{tobias.woehrer@gmail.com}}

\thanks{{\it Research supported by the Austrian Science Fund (FWF) under Grant No.\ Y330}}
\thanks{Monatsh. Math. {\bf 179}, 603--613 (2016).}


\begin{abstract}
We derive an asymptotic expansion for the Weyl function of a one-dimensional
Schr\"odinger operator which generalizes the classical formula by Atkinson.
Moreover, we show that the asymptotic formula can also be interpreted in the sense of
distributions.
\end{abstract}

\maketitle

\section{Introduction}\label{s1}

The $m$-function or Weyl--Titchmarsh function introduced by Weyl in \cite{Weyl10} plays a fundamental role in spectral theory for Sturm--Liouville operators. In particular, it is known that in the case of sufficiently "nice" potentials $q$ all information about the spectral properties of self-adjoint realizations of the differential expression
\be
 -\frac{d^2}{dx^2} + q(x),
\ee
acting in $L^2(0,\infty)$, are encoded in this function. In 1952 Marchenko proved (see \cite[Theorem 2.2.1]{Mar52}) that the $m$-function corresponding to the Dirichlet boundary condition at $x=0$ behaves asymptotically at infinity like the corresponding function of the unperturbed operator (corresponding to $q\equiv0$), that is,
\be\label{eq:mar}
m(z) = -{\sqrt{-z}}\,(1+\oo(1)),
\ee
as $z\to \infty$ in any nonreal sector in the open upper complex half-plane $\C_+$ (let us stress that the high-energy behavior of $m$ can be deduced from the asymptotic behavior of the corresponding spectral function, see \cite[Theorem~II.4.3]{LS88}). A simple proof of this formula was found by Levitan in \cite{Lev52} (a short self-contained proof of \eqref{eq:mar} can be found in, e.g., \cite[Lemma~9.19]{Te14}). Since then the high-energy asymptotics $z\to\infty$ of the $m$-function received
enormous attention over the past three decades as can be inferred, for instance from
\cite{At81}, \cite{Be88}, \cite{Be89}, \cite{BM97}, \cite{DL91},
\cite{Ev72}, \cite{EH78}, \cite{EHS83}, \cite{Ha83}--\cite{Ha87},
\cite{Hi69}, \cite{HKS89}, \cite{KK86}, \cite{KK87}, \cite{Ry01}, \cite{Si98}
and the references therein.

Typically there are two directions which are of interest: If one assumes $q$ to be smooth a full asymptotic expansion can
be given. Otherwise, one tries to derive the leading asymptotic under minimal assumptions on $q$. One of the key
improvements in this latter direction is due to Atkinson \cite{At81} who showed
\be\label{eq:at}
m(z) = -\sqrt{-z} - \int_{[0,x_0)} \E^{-2\sqrt{-z}y} q(y) dy + o(z^{-1/2})
\ee
for arbitraty $x_0\in(0,\infty)$.
In particular, if $0$ is a Lebesgue point of $q$ this implies
\be\label{eq:at2}
m(z) = -\sqrt{-z} - \frac{q(0)}{2\sqrt{-z}} + o(z^{-1/2}).
\ee
On the other hand, the case of a locally integrable potential does not cover the case where
$q$ is a single Dirac $\delta$ one of the most popular toy models which can be found in any
text book on quantum mechanics. Even though the case of delta potentials has a long
tradition (see e.g.\ the monograph \cite{aghh}) the case where $q$ is replaced by an
arbitrary measure got significant interest only recently and we refer to \cite{BaRe05},
\cite{EGNT12}--\cite{EGNT12b}, \cite{ET12}, \cite{Pe13}, \cite{SS00} and the literature therein. 

Our question in the present paper is to what extent \eqref{eq:at} remains valid
when $q$ is replaced by a measure. Moreover, we will also show that \eqref{eq:at2}
remains true when interpreted in the sense of distributions.

\section{Schr\"odinger Operators with Measure-Valued Coefficients} \label{s2}

Our main object are one-dimensional Schr\"odinger operators in the Hilbert space $L^2(a,b)$, $-\infty< a < b \le \infty$,
associated with the differential expressions
\be
\tau f = \left( - f' + \int f\, d\chi \right)',
\ee
where $\chi$ is a locally finite signed Borel measure on $[a,b)$. In particular, we assume that
$\tau$ is regular at $a$, that is, $a\in\R$ and the total variation of $\chi$ is finite near $a$ (i.e.,
$|\chi|([a,x_0)) <+\infty$ for every $x_0\in(a,b)$).

The maximal domain of this differential expression is given as
\[
\Deftau = \left\{ f\in \ACloc \,|\, \left ( x\mapsto -f'(x) + \int f\, d\chi \right) \in \ACloc \right\},
\]
which leads to a jump condition for $f'(x)$ at every point mass,
\be
f'(x+)-f'(x-) = \chi(\{x\}) f(x).
\ee
We fix $f'(x)$ to be left continuous. At $x=a$ the above condition has to be understood
as the definition of the left limit.

In order to get a self-adjoint operator we look at the corresponding maximal operator
associated with $\tau$ in $L^2(a,b)$ with the domain
\begin{align*}
&\dom\Tmax = \{ f\in \Deftau \, | \, f,\tau f \in L^2(a,b) \}.
\end{align*}
For $f,g\in \dom\Tmax$ we can define the Wronskian as usual
\be
W_x(f,g) = f(x) g'(x) - f'(x) g(x)
\ee
and one can verify the Lagrange identity
\be
\int_{[c,d)} (g \tau f - f\tau g) dx = W_d(f,g) - W_c(f,g)
\ee
where $x, c,d$ include the interval endpoints as one-sided limits.
In particular, the Wronskian is constant for two solutions of $\tau u = z u$.

We say $\tau$ is in the limit-circle (l.c.) case at $b$ if all solutions of $\tau u= z u$ are square integrable near $b$ and we say that $\tau$ is in the limit-point (l.p.) case at $b$ otherwise.

To obtain a self-adjoint operator from $\Tmax$ we will choose appropriate boundary conditions. First of
all we will choose a Dirichlet boundary condition at $a$. Then, if $\tau$ is in the l.p.\ case at $b$, no
further boundary condition is needed and the corresponding operator
\[
\dom{S} = \{ f\in \dom\Tmax \,| \, f(a) = 0 \}
\] 
is a self-adjoint restriction of $\Tmax$. Otherwise, if $\tau$ is in the l.c.\ case at $b$, we need an additional boundary
condition at $b$ in which case every restriction of $\Tmax$ with domain
\[
\dom{S} = \left\{ f\in \dom\Tmax \, | \, f(a) = 0, W_b(f, w^*)=0 \right\},
\]
where $w\in \dom\Tmax$ satisfies $W_b(w,w^*)=0$ and $W(h,w^*)\neq 0$ for some $h\in\dom\Tmax$, is a self-adjoint operator.

We refer to \cite{ET12} for background and general theory.

\section{Asymptotics for the Weyl function} \label{s3}

In this section we will assume that the left endpoint $a$ is regular and without loss of
generality we will assume $a=0$. To simplify notation we denote 
\[
\chi(x):=\begin{cases}
\chi([0,x)), & x\in(0,b),\\
0, & x=0.
\end{cases}
\]
In this case we have a basis
of solutions $c(z,x)$, $s(z,x)$ of $\tau u = z u$ determined by the initial conditions
\be
c(z,0)=1,\: c'(z,0)=0, \qquad s(z,0)=0,\: s'(z,0)=1,
\ee
such that $W(c(z),s(z))=1$. Here and in what follows a prime will always denote
a derivative with respect to the spatial coordinate $x$.
They are given as the solutions of the following integral equations
\begin{align}\label{eq:iec}
c(z,x) =& \cosh(\sqrt{-z} x) + \frac{1}{\sqrt{-z}} \int_{[0,x)} \sinh(\sqrt{-z}(x-y)) c(z,y) d\chi(y),\\
s(z,x) =& \frac{1}{\sqrt{-z}} \sinh(\sqrt{-z} x) + \frac{1}{\sqrt{-z}} \int_{[0,x)} \sinh(\sqrt{-z}(x-y)) s(z,y) d\chi(y).
\end{align}
In fact, this can be verified using integration by parts, which also shows
\begin{align}
c'(z,x) =& \sqrt{-z}\sinh(\sqrt{-z} x) + \int_{[0,x)} \cosh(\sqrt{-z}(x-y)) c(z,y) d\chi(y),\\ \label{eq:iesp}
s'(z,x) =& \cosh(\sqrt{-z} x) + \int_{[0,x)} \cosh(\sqrt{-z}(x-y)) s(z,y) d\chi(y).
\end{align}
Here and in what follows $\sqrt{\cdot}$ will always denote the standard branch of the square root with branch
cut along $(-\infty,0)$.

We will need their high-energy asymptotics as $\im(z)\to\infty$.

\begin{lemma}\label{lem:asymcs}
The function $c(z,x)$ and its derivative $c'(z,x)$ can be written as
\begin{align}\nn
c(z,x)=& \cosh(\sqrt{-z}x) + \frac{1}{2 \sqrt{-z}}\sinh(\sqrt{-z}x) \chi(x)\\ \nn
& {} +\frac{\E^{\sqrt{-z}x}}{4\sqrt{-z}} \left (\int_{[0,x)} \E^{-2\sqrt{-z}y} d\chi(y)
- \int_{[0,x)} \E^{-2\sqrt{-z}(x-y)}d\chi(y)\right)\\
& {} - \frac{\E^{\sqrt{-z}x}}{z}E_1(z,x),\\\nn 
c'(z,x) =& \sqrt{-z} \sinh(\sqrt{-z}x) + \frac{1}{2} \cosh(\sqrt{-z}x) \chi(x)\\ \nn
& {} + \frac{\E^{\sqrt{-z}x}}{4} \left( \int_{[0,x)} \E^{-2\sqrt{-z}y} d\chi(y)
+ \int_{[0,x)} \E^{-2\sqrt{-z}(x-y)} d\chi(y) \right)\\
& {} + \frac{\E^{\sqrt{-z}x}}{\sqrt{-z}} E_2(z,x),
\end{align}
with error functions $E_j(z,x)$ satisfying $|E_j(z,x)|\le C|\chi|([0,x))$ and
\be
E_j(z,x)= \frac{1}{8} \int_{(0,x)}\big(\chi(y) + \chi(\{0\})\big) d\chi(y) + \oo(1), \quad j=1,2,
\ee
as $\im(z)\to +\infty$.

Similarly, the function $s(z,x)$ and its derivative $s'(z,x)$ can be written as
\begin{align}\nn
s(z,x) =& \frac{ 1}{\sqrt{-z}}\sinh(\sqrt{-z}x) - \frac{1}{2z} \cosh(\sqrt{-z}x) \chi(x)\\ \nn
& {}+\frac{\E^{\sqrt{-z}x}}{4z}\left(\int_{[0,x)} \E^{-2\sqrt{-z}y} d\chi(y) + \int_{[0,x)} \E^{-2\sqrt{-z}(x-y)} d\chi(y)\right)\\
& {} +  \frac{\E^{\sqrt{-z}x}}{\sqrt{-z}^3}E_3(z,x),\\ \nn
s'(z,x) =& \cosh(\sqrt{-z}x) + \frac{1}{2\sqrt{-z}} \sinh(\sqrt{-z}x) \chi(x)\\ \nn
& {} - \frac{\E^{\sqrt{-z}x}}{4\sqrt{-z}} \left(\int_{[0,x)} \E^{-2\sqrt{-z}y} d\chi(y) - \int_{[0,x)} \E^{-2\sqrt{-z}(x-y)}  d\chi(y) \right)\\
& {} - \frac{\E^{\sqrt{-z}x}}{z}E_4(z,x),
\end{align}
with error functions $E_j(z,x)$ satisfying $|E_j(z,x)|\le C|\chi|([0,x))$ and
\be
E_j(z,x)= \frac{1}{8} \int_{(0,x)}\big(\chi(y) - \chi(\{0\})\big) d\chi(y) + \oo(1), \quad j=3,4,
\ee
as $\im(z)\to +\infty$.
\end{lemma}

\begin{proof}
Abbreviate $k=\sqrt{-z}$ and note $\re(k)\ge 0$.
First of all, considering the function $\ti{c}(z,x)= \E^{-k x} c(z,x)$ we look at the corresponding integral equation 
\[
\ti{c}(z,x)= \frac{1+\E^{-2kx}}{2} + \int_{[0,x)} \frac{1-\E^{-2k(x-y)}}{2k} \ti{c}(z,y) d\chi(y)
\]
from which it follows that there is a bounded solution satisfying
\[
|\ti{c}(z,x)| \le \exp\big(|k|^{-1} |\chi|([0,x))\big)
\]
by using the usual iteration scheme (cf.\ \cite[Theorem A.2]{ET12}).
Now we use bootstrapping and insert this information into our integral equation. First the integral equation for $c(z,x)$ can be written as
\be \label{eqn:E1}
c(z,x) = \cosh(k x) + \frac{\E^{k x}}{k}\ti{E}_1(z,x),
\ee
with the error term
\[
\ti{E}_1(z,x)=\int_{[0,x)} \frac{1-\E^{-2k (x-y)}}{2} \ti{c}(z,y) d\chi(y)
\]
which is is locally uniformly bounded in $x$ by the above estimate for $\ti{c}(z,x)$ . Reinserting \eqref{eqn:E1} into the integral equation for $c(z,x)$ leads to the desired representation of the solution $c(z,x)$, where the error term
\[
E_1(z,x)= \int_{[0,x)} \frac{1-\E^{-2k (x-y)}}{2} \ti{E}_1(z,y) d\chi(y)
\]
is locally uniformly bounded in $x$. 

To compute the desired estimate for the error term $E_1(z,x)$ we insert \eqref{eqn:E1} into the definition of $\ti{E}_1(z,x)$, which leads to
\[
\ti{E}_1(z,x)=
\begin{cases}
\frac{1}{4}(\chi(x) + \chi(\{0\})) + \OO(\frac{1}{2k}), & x>0,\\
0, & x=0,
\end{cases}
\]
by the dominated convergence theorem, where the estimate is locally uniform in $x$ as $\im(z)\to +\infty$. Now inserting this estimate into the definition of  $E_1(z,x)$ and applying the dominated convergence theorem again, leads to the desired estimate for the error term $E_1(z,x)$.

Similarly, considering $\ti{s}(z,x)= k\E^{-k x} s(z,x)$ we look at the corresponding integral equation 
\[
\ti{s}(z,x)= \frac{1-\E^{-2kx}}{2} + \int_{[0,x)} \frac{1-\E^{-2k(x-y)}}{2k} \ti{s}(z,y) d\chi(y)
\]
and conclude that there is a bounded solution satisfying
\[
|\ti{s}(z,x)| \le \exp\big(|k|^{-1} |\chi|([0,x))\big).
\]
The rest follows as before.
\end{proof}

Next we recall the Weyl function $m(z)$ defined such that
\be\label{eqn:weylfkt}
u(z,x) = c(z,x) + m(z) s(z,x), \qquad z\in\C\setminus\R,
\ee
is square integrable near $b$ and satisfies the boundary condition of our operator at $b$ (if there is one).
Following the original approach of Weyl we recall the Weyl circles with center, radius given by
\be
q(z,x_0)= - \frac{W_{x_0}(c(z,\cdot),s(z,\cdot)^*)}{W_{x_0}(s(z,\cdot),s(z,\cdot)^*)}, \qquad
r(z,x_0)= \frac{1}{|W_{x_0}(s(z,\cdot),s(z,\cdot)^*)|},
\ee
with $x_0\in[0,b)$, respectively. By construction the solutions $c(z,x) + m\, s(z,x)$ with $m$ on the Weyl circle are
precisely the ones which satisfy a real boundary condition at $x_0$:
\be\label{eqn:quotient}
\frac{c'(z,x_0) + m\, s'(z,x_0)}{c(z,x_0) + m\, s(z,x_0)} \in \R\cup\{\infty\}.
\ee
Taking $x_0\nearrow b$ these circles are nested and hence converge to a circle (limit circle case) or to
a point (limit point case). In the first case, the points on the circle correspond to the Weyl functions
corresponding to different self-adjoint realizations and in the second case the point corresponds to
the unique Weyl function of the unique self-adjoint realization.

Moreover, for $\im(z)>0$, those where the quotient in \eqref{eqn:quotient} is in the upper, lower half-plane are those for which $m$ is
in the interior, exterior of the Weyl circle, respectively. Hence, if we find an $m$ in the interior,
the distance between $m$ and $m(z)$ can be at most $2 r(z,x_0)$. This is precisely the idea (due to \cite{At81})
of the following lemma:

\begin{lemma}
For every $x_0\in(0,b)$ we have
\be\label{eq:asmv}
m(z) = -\frac{c(z,x_0)\sqrt{-z}+c'(z,x_0)}{s'(z,x_0)+\sqrt{-z}s(z,x_0)} + \OO(z\E^{-2\sqrt{-z} x_0}),
\ee
as $\im(z)\to+\infty$, where the error depends only on the total variation $|\chi|([0,x_0))$.
Moreover,
\be
\frac{c(z,x)\sqrt{-z}+c'(z,x)}{s'(z,x)+\sqrt{-z}s(z,x)}
= \sqrt{-z} \frac{1+\frac{1}{\sqrt{-z}} \int_{[0,x)} \ti{c}(z,y) d\chi(y)}{1+\frac{1}{\sqrt{-z}} \int_{[0,x)} \ti{s}(z,y) d\chi(y)}
\ee
where $\ti{c}(z,x)= \E^{-\sqrt{-z} x} c(z,x)$ and $\ti{s}(z,x)= \sqrt{-z}\E^{-\sqrt{-z} x} s(z,x)$.
\end{lemma}

\begin{proof}
For $\im(z)>0$ it follows that the solution defined via the initial condition
\[
v(z,x_0)=1, \qquad v'(z,x_0)= -\sqrt{-z}
\]
with $x_0\in(0,b)$ corresponds to a point in the interior of the Weyl circle. Indeed we have
\[
\frac{v'(z,0)}{v(z,0)} = \frac{W_0(c(z,\cdot),v(z,\cdot))}{W_0(v(z,\cdot),s(z,\cdot))}
\]
and the constancy of the Wronskian implies
\be\label{eqn:wron}
\frac{v'(z,0)}{v(z,0)} = \frac{W_{x_0}(c(z,\cdot),v(z,\cdot))}{W_{x_0}(v(z,\cdot),s(z,\cdot))}= -\frac{c(z,x_0)\sqrt{-z}+c'(z,x_0)}{s'(z,x_0)+\sqrt{-z}s(z,x_0)}.
\ee
Now an easy computation shows that
\[
\frac{c'(z,x_0) + s'(z,x_0) \frac{v'(z,0)}{v(z,0)}}{c(z,x_0) + s(z,x_0) \frac{v'(z,0)}{v(z,0)}} = -\sqrt{-z} \in \C_+.
\]
Hence the point $\frac{v'(z,0)}{v(z,0)}$ lies in the interior of the Weyl circle by the considerations prior to this lemma. As the same is true for the Weyl function $m(z)$ of our problem, we obtain
\[
m(z) = \frac{v'(z,0)}{v(z,0)} + \OO(r(z,x_0)) = \frac{v'(z,0)}{v(z,0)} + \OO(z\E^{-2\sqrt{-z} x_0})
\]
as $\im(z)\to\infty$, where we have used Lemma~\ref{lem:asymcs} for the second identity.

The last part is a straightforward calculation using \eqref{eq:iec}--\eqref{eq:iesp}.
\end{proof}

Combining this lemma with Lemma~\ref{lem:asymcs} gives our main result:

\begin{theorem}
For every $x_0\in(0,b)$ the Weyl $m$-function has the asymptotic behavior
\be\label{eq:main}
m(z) = -\sqrt{-z} - \int_{[0,x_0)} \E^{-2\sqrt{-z}y} d\chi(y) + \oo(z^{-1/2})
\ee
as $\im(z)\to\infty$. Moreover, the error satisfies an estimate of the type $|\oo(z^{-1/2})| \le C |z|^{-1/2}$,
where the constant depends only on the total variation $|\chi|([0,x_0))$.
\end{theorem}

\begin{proof}
By inserting Lemma~\ref{lem:asymcs} into the identity \eqref{eqn:wron} a long but straightforward
computation shows that
\[
m(z)= -k - I_1 - \frac{1}{k} \left(E_1 + E_2 - E_3 - E_4\right) + \frac{1}{2k} (\chi(x_0) I_1-I_1^2) + \oo(\frac{1}{k})
\]
as $\im(z)\to+\infty$, where we abbreviated $k=\sqrt{-z}$ and $I_1(z)=\int_{[0,x_0)} \E^{-2ky}d\chi(y)$.
Inserting the estimates for the error terms $E_j(z)=E_j(z,x_0)$ of Lemma~\ref{lem:asymcs} as well as the estimate
\[
I_1(z) =  \chi(\{0\})+ \oo(1)
\]
as $\im(z)\to+\infty$, finally proves the theorem.
\end{proof}

\begin{remark}
(i). We want to emphasize that in contradistinction to \cite{At81} our approach is more direct and avoids the
use of Riccati equations for the Weyl function. In addition to being simpler this approach also retains
a good control over the error with respect to the total variation of $\chi$. This will turn out crucial
for our following application which states that \eqref{eq:at2} continues to hold in the sense of
distributions. A similar result (for Neumann boundary conditions) can be found in Lemma~5.1 of \cite{BaRe05}
with a weaker error term and again without the above mentioned control.

\noindent
(ii). It is also possible to allow for more general potentials. In fact, one could consider potentials in $H^{-1}_{loc}$,
however, in this case Lemma~\ref{lem:asymcs} is expected to break down since $\chi(x)$ will be in $L^2$ and
hence there are no point values. We refer to Theorem~B.2 in \cite{EGNT12b}, where the weaker result
$m(z) = -\sqrt{-z} +o(\sqrt{-z})$ for in fact a slightly lager class than $H^{-1}_{loc}$ is shown.

\noindent
(iii). For an arbitrary left endpoint $a$ equation \eqref{eq:main} reads
\[
m(z) = -\sqrt{-z} - \int_{[a,x_0)} \E^{-2\sqrt{-z}(y-a)} d\chi(y) + \oo(z^{-1/2}).
\]
(iv). Of course one can iterate this procedure to get further terms in the above expansion.
For example using one more step one obtains:
\begin{align*}
m(z) =& -k - \int_{[0,x_0)} \E^{-2k y} d\chi(y)\\
& {} - \frac{1}{2k} \int_{[0,x_0)} \left((1-\E^{-2k (x_0-y)})  \int_{[0,y)} \E^{-2k r}d\chi(r) \right) d\chi(y)\\
& {} + \frac{1}{2k} \int_{[0,x_0)} (1-\E^{-2k y}) d\chi(y) \int_{[0,x_0)} \E^{-2ky}d\chi(y) + \OO(k^{-2}).
\end{align*}
\end{remark}

\begin{theorem}
Denote by $m(z,t)$ the Weyl function associated with our operator restricted to the interval $(t,b)$
with a Dirichlet boundary condition at $t\in[a,b)$ and keeping the boundary condition at $b$ (if any) fixed.
Then for any test function $\phi \in C^\infty_c(a,b)$ we have
\be
\int_a^b m(z,t) \phi(t) dt = -\sqrt{-z} \int_a^b \phi(t) dt - \frac{1}{2\sqrt{-z}} \int_a^b \phi(s) d\chi(s) + \oo(z^{-1/2}).
\ee
\end{theorem}

\begin{proof}
All we have to do is multiply \eqref{eq:main} with $\phi$ and integrate with respect to $t$.
By our bound on the error term we can integrate the error term using dominated convergence
and the rest follows by Fubini:
\begin{align*}
\int_a^b m(z,t) \phi(t) dt 
&=  -\sqrt{-z} \Phi_0 - \int_{\R^2} \phi(t) \id_{[t,t+\eps)}(s) \E^{-2\sqrt{-z}(s-t)} d\chi(s) dt +\oo(z^{-1/2})\\
&=  -\sqrt{-z} \Phi_0 - \int_{\R^2} \phi(t) \id_{(s-\eps,s]}(t) \E^{2\sqrt{-z}(t-s)} dt\, d\chi(s)  +\oo(z^{-1/2})\\
&= -\sqrt{-z} \Phi_0 - \frac{1}{2\sqrt{-z}} \int_a^b \phi(s) d\chi(s) + \oo(z^{-1/2}),
\end{align*}
where we have abbreviated $\Phi_0 = \int_a^b \phi(t) dt$ and $\id_\Omega$ denotes the indicator function
of a set $\Omega$. Moreover, in the last step we have used
\[
\int_{s-\eps}^s \phi(t) \E^{2\sqrt{-z}(t-s)} dt = \frac{1}{2\sqrt{-z}} \phi(s) + O(z^{-1}),
\]
which follows from a simple integration by parts.
\end{proof}

Finally, we look at the example

\begin{example}
Denote by $\delta_0$ a single Dirac delta measure located at $x=0$ and set
\be
\chi = \alpha \delta_0, \qquad \alpha\in\R.
\ee
In this case the solution $u(z,x)$ from \eqref{eqn:weylfkt} is given as $u(z,x)=\E^{-\sqrt{-z}x}$ and thus $u(z,0)=1$ and
$u'(z,0)=u'(z,0-)= -\sqrt{-z} -\alpha$ implying
\be
m(z,0)= -\sqrt{-z} -\alpha
\ee
in agreement with \eqref{eq:main}.
\end{example}

\bigskip

\noindent 
{\bf Acknowledgments.} We are indebted to Jonathan Eckhardt, Fritz Gesztesy, and Helge Holden for discussions on this subject.


\end{document}